\documentclass[10pt,reqno]{amsart}
\usepackage[comma,numbers,sort]{natbib}
\usepackage{amsmath}
\usepackage{amsfonts}
\usepackage{amssymb}
\usepackage{amsthm}
\usepackage{dsfont}
\usepackage{bbm}
\usepackage{hyperref}
\usepackage{paralist}
\usepackage{enumitem}
\usepackage{graphicx}
\usepackage[margin=1.3in]{geometry}
\usepackage{cleveref}
\usepackage{pgfplots}
\usepackage{tikz}
\usepackage{tikz-dimline}
\usepackage{subcaption}
\usepackage[normalem]{ulem}
\usepackage[font=small,labelfont=bf]{caption}
\usepackage{multicol}
\usepackage{adjustbox}
\usepackage{multirow}
\pgfplotsset{compat=newest}
\tikzset{no slope/.code={\pgfslopedattimefalse}}
\tikzset{
    convexset/.style = {line width = 0.75 pt, fill = lightgray},
    ext/.style = {circle, inner sep=0pt, minimum size=2pt, fill=black},
    segment/.style = {line width = 0.75 pt}
        }
\colorlet{lightgray}{gray!50}

\newcommand\set[2]{\left\{ {#1} \ : \ {#2} \right\}}

\newcommand\R{\mathbb{R}}
\newcommand\Z{\mathbb{Z}}

\newcommand\x{\mathbf{x}}
\newcommand\y{\mathbf{y}}
\newcommand\n{\mathbf{n}}

\renewcommand\int{\mathrm{int}\,}

\newcommand\define{\,:=\,}

\newcommand{\review}[1]{{\color{red}{#1}}}
\newtheorem{thm}{Theorem}
\newtheorem{mainthm}{Theorem}
\newtheorem*{thm*}{Theorem}

\newtheorem{lem}[thm]{Lemma}
\newtheorem{prop}[thm]{Proposition}
\newtheorem{conj}[thm]{Conjecture}
\theoremstyle{definition}
\newtheorem{dfn}{Definition}

\newtheorem{rem}{Remark}

\newcommand\ignore[1]{}

\newcommand{\conv}{\textup{conv}}
\title[A few more Lonely Runners]{A few more Lonely Runners}
\author{Avinash Bhardwaj}
\address{Avinash Bhardwaj,Indian Institute of Technology Bombay, Mumbai, India 400076}
\email{abhardwaj@iitb.ac.in}
\author{Vishnu Narayanan}
\address{Vishnu Narayanan, Indian Institute of Technology Bombay, Mumbai, India 400076}
\email{vishnu@iitb.ac.in}
\author{Hrishikesh Venkataraman}
\address{Hrishikesh Venkataraman, Indian Institute of Science Education and Research, Pune, India 411008}
\email{hrishikesh.v@students.iiserpune.ac.in}

\thanks{}

\begin{document}
\begin{abstract}
\emph{Lonely Runner Conjecture}, proposed by J\"{o}rg M. Wills and so nomenclatured by Luis Goddyn, has been an object of interest since it was first conceived in 1967: Given positive integers $k$ and $n_1,n_2,\ldots,n_k$ there exists a positive real number $t$ such that the distance of $t\cdot n_j$ to the nearest integer is at least $\frac{1}{k+1}$, $\forall~~1\leq j\leq k$. In a recent article, Beck, Hosten and Schymura described the \textit{Lonely Runner polyhedron} and provided a polyhedral approach to identifying families of lonely runner instances. We revisit the \textit{Lonely Runner polyhedron} and highlight some new families of instances satisfying the conjecture. In addition, we relax the sufficiency of existence of an integer point in the \emph{Lonely Runner polyhedron} to prove the conjecture. Specifically, we propose that it suffices to show the existence of a lattice point of certain superlattices of the integer lattice in the \emph{Lonely Runner polyhedron}.
\end{abstract}
\maketitle

\section{Introduction}
The following was conjectured independently by Wills \cite{wills1967zwei} and Cusick \cite{cusick1974view}.
\begin{conj}
Consider $k+1$ pairwise-distinct real numbers $n_0, n_1,\ldots, n_k$. For any $0\leq i\leq k$, there is a real number $t$ such that the distance of $t(n_j-n_i)$ to the nearest integer is at least $\dfrac{1}{k+1}$, for all $0\leq j\leq k$, $j\neq i$.
\end{conj}
Goddyn \cite{bienia1998flows} termed the above as the \emph{Lonely Runner Conjecture}, where he considered the real numbers to be the speeds of runners on a unit length circular track, all of whom start at the same time and at the same point on the track. Note that it is the $(n_j-n_i)$ that matters and not the $n_i$ itself. Consequently, the speed of a runner can be subtracted from all the speeds. This leaves a runner stationary and possibly, some runners with negative speeds. However, running with a negative speed is equivalent to running in the opposite direction, with the same magnitude of speed. Combining the stationarity of a runner with the symmetry of the track, all resultant speeds can be assumed to be positive. A similar argument can be made for the time $t$ as well. Thus, an equivalent version of the conjecture is:
\begin{conj} \textup{(Lonely Runner Conjecture)}\label{conj:lonelyrunner2}
Given $k$ positive real numbers $n_1, n_2,\ldots, n_k$, there is a non-negative real number $t$ such that the distance of each $tn_i,~1\leq i \leq k$ to its nearest integer is at least $\dfrac{1}{k+1}$.
\end{conj}
\ignore{Since we say $t$ is a real number, one could question as to what negative time means. Motion under negative time can be thought of as motion with all speeds negated (i.e, all runners run in the opposite direction, with the same magnitude of speed as earlier).}
In the remainder of this paper, we refer to the \emph{Lonely Runner Conjecture} in the form described in \Cref{conj:lonelyrunner2}. Furthermore, we refer to any $\n=(n_1,n_2,\ldots,n_k)$ as a \emph{lonely runner instance} if $\n$ satisfies the conjecture.

With minimal effort, the conjecture can be proven to be true for $k=1$. The $k=2$ case was proven in the context of Diophantine approximation by Wills \cite{wills1967zwei}. Using a similar approach, the $k=3$ case was proven, first by Betke and Wills \cite{betke1972untere}, and then by Cusick \cite{cusick1974view}. The first proof for $k=4$ was given using a view-obstruction problem approach, by Cusick and Pomerance \cite{cusick1984view}. A simpler proof was then given by Bienia et al. \cite{bienia1998flows} by relating the conjecture to nowhere zero flows in regular matroids. Thinking of the conjecture as a covering problem, the $k=5$ case was proven by Bohman, Holzman and Kleitman \cite{bohman2001six}. Later, Renault \cite{renault2004view} provided a simpler proof for this, using a case-by-case analysis of all possible speed vectors. Subsequently, a study of the regular chromatic number of distance graphs, by Barajas and Serra \cite{barajas2008lonely}, led to a proof for $k=6$. The conjecture remains open for $k\geq 7$.

Henze and Malikiosis \cite{henze2017covering} were among the first ones to study the conjecture from a polyhedral theory perspective. They made use of the equivalence of - (i) well-known geometric problems, such as the motion of billiards balls in a cube avoiding an inner cube, and (ii) determining the covering radii of lattice zonotopes. \ignore{Consider the path traversed by a billiards ball constrained to move inside a unit cube. The authors first gave the relation between - (a) minimum size of an inner cube placed inside the unit cube such that the path of any billiards ball intersects the inner cube, and (b) covering radius of a corresponding zonotope. \review{Using this, they provide a reformulation of the conjecture in terms of the maximum dilation of a zonotope such that, for a new lattice $\Lambda$, the zonotope contains a nontrivial lattice point $p\in\Lambda$, $p\notin\mathbb{Z}^k$.}} They established, as the first unconditional proof, that it suffices to prove the conjecture for positive integral speeds.

An alternative approach to fixing $k$ and attacking the problem is to identify families of speed vectors that are \emph{lonely runner instances}. Such an endeavor was taken by Beck, Hosten and Schymura \cite{polyhedra2019matthias}. The authors introduced the \emph{Lonely Runner polyhedron} and provided families of \emph{lonely runner instances}, one in terms of the speeds of the fastest and slowest runners, and another in terms of the parities of the runners' speeds. Subsequently, by considering the projection of the polyhedron on to the first one and two coordinates, they provide families of \emph{lonely runner instances} by relating the speeds of the second and third fastest runners with the slowest.

Very recently, Rifford \cite{rifford2022time} employed a new method to approach the conjecture. In particular, he explored the problem of determining an upper bound on $t$, in terms of the number of rounds covered by the slowest runner, and described this in terms of a covering problem. A conjecture was proposed for the covering problem and it was proven to be true for all $2 \leq k \leq 5$.

One related quantity, to \Cref{conj:lonelyrunner2}, that has been well studied in the past is the \emph{gap of loneliness}. Let $d(x,\mathbb{Z})$ represent the distance of a real number $x$ to the nearest integer.\ignore{Given $\n \in \Z^k$, define
\begin{equation*}
\Gamma(\n)=\sup\limits_{t\in\mathbb{R}} \min\limits_{i\in[k]} d(tn_i,\mathbb{Z})
\end{equation*}}
The \emph{gap of loneliness} is then defined as 
$$
\Gamma(k)=\inf_{\n\in\mathbb{Z}^k}\sup\limits_{t\in\mathbb{R}} \min\limits_{i\in[k]} d(tn_i,\mathbb{Z})
$$
An equivalent version of the \emph{Lonely Runner Conjecture} states that $\Gamma(k)=\dfrac{1}{k+1}$. Dirichlet approximation theorem establishes that $\Gamma(k)\leq\dfrac{1}{k+1}$. Several efforts have been made to prove the converse in \cite{chen1994view}, \cite{chen1999view}, \cite{perarnau2016correlation},\cite{rifford2022time} and \cite{tao2018some}, among others. \ignore{The present benchmark is $\frac{1}{2k-1+\frac{1}{2k-3}}$ for all $k$ and $\left(\frac{1}{2k}+o\left(\frac{\log k}{k^2(\log\log k)^2}\right)\right)$ for large $k$.} Using this version of the conjecture, Tao \cite{tao2018some} proved that it suffices to show that the \emph{Lonely Runner Conjecture} holds for $\n$ having all components at most $k^{O(k^2)}$. In an attempt to prove the same, Tao proved that if all the components of a speed vector $\n$ are at most $1.2k$, then $\n$ is a \emph{lonely runner instance}. Furthermore, it was noted that a desired goal is to increase the multiplier $1.2$ to $2$. Recently, Bohman and Peng \cite{bohman2022coprime} proved the approximate version of this result when they showed that it is possible to get the multiplier arbitrarily close to $2$ when the number of runners is sufficiently large.

Assuming that $n_1\geq n_2\geq \ldots\geq n_k$, every $\mathbf{n}$ is associated with a \emph{lacunarity} value. To prove the \emph{Lonely Runner Conjecture}, it suffices to prove that every $1$-lacunary integer sequence is a \emph{lonely runner instance}.  In one of the first results with this perspective, Pandey\cite{pandey2009note} proved that $\frac{2(k+1)}{k-1}$-lacunary integer sequences are \emph{lonely runner instances}. Barajas and Serra\cite{barajas2009chromatic} improved the result when they showed that all $2$-lacunary integer sequences are \emph{lonely runner instances}. Dubickas\cite{dubickas2011lonely} established that for very large values of $k$, $\left(1+\frac{33\log(k)}{k}\right)$-lacunary integer sequences are \emph{lonely runner instances}. Subsequently, Czerwi\'{n}ski\cite{czerwinski2018lonely} showed that $\n$ is a \emph{lonely runner instance} if $(\n\setminus\{n_1\})$ is $\frac{k+1}{8e}$-lacunary. Recently, it was proven that $\n$ is a \emph{lonely runner instance} if $(\n\setminus\{n_1,n_k\})$ is $\frac{2k}{k-1}$-lacunary and $GCD(n_{k-1},n_k)\leq \frac{k-1}{k+1}(n_{k-1}-n_k)$ \cite{polyhedra2019matthias}.
 
For a long time, there have been thoughts on whether the \emph{Lonely Runner Conjecture} has been put forth in the most general way possible. As per \cite{polyhedra2019matthias}, Wills recently Conjectured that the runners need not start from a common point, and there would still be a time when all of them are at least $\frac{1}{k+1}$ distance from a fixed origin. This statement, as with the \emph{Lonely Runner Conjecture} is easy to prove for $k=1$. Along with the new statement, a proof was provided for $k=2$ by Beck et al. \cite{polyhedra2019matthias}. Cslovjecsek et al. \cite{cslovjecsek2022computing} studied the covering radius of polytopes and proved that the new statement holds for $k=3$. More recently, Rifford\cite{rifford2022time} defined the view-obstruction version for the new statement and provided new proofs for $k=2,3$.
\ignore{Apart from the various approaches mentioned above, work has been done to study a modification of the conjecture and generalizing results to the version that we study here. The modified version allows the runners to have arbitrary starting points on the track. This has been a subject of interest in, for example, \cite{cslovjecsek2022computing} and \cite{rifford2022time}.}

We summarize the aforementioned results in Table \ref{tab:casesresolved}.
\begin{table}[ht]
\begin{tabular}{|c|c|c|} 
\hline
Property & Resolution criteria & Cases resolved \\
\hline
Number of runners & $k\in\mathbb{N}^+$ & $k\leq 6$\\
\hline
Speeds of runners & $n_1\geq n_2\geq \ldots \geq n_k$ & $n_1\leq kn_k$\\
\hline
\multirow{2}{*}{Lacunarity, $\lambda(\n)$} & \multirow{2}{*}{$\lambda(\n) \geq 1$} & $\lambda(\n) \geq  2$\\
\cline{3-3}
&  & $\lambda(\n) \geq \left(1+\frac{33\log(k)}{k}\right)$ \text{for large $k$}\\
\hline
\multirow{2}{*}{Gap of Loneliness, $\Gamma(k)$} & \multirow{2}{*}{$\Gamma(k) \geq \frac{1}{k+1}$} & $\Gamma(k) \geq \frac{1}{2k-1+\frac{1}{2k-3}}$\\
\cline{3-3}
& & $\Gamma(k) \geq\left(\frac{1}{2k}+o\left(\frac{\log k}{k^2(\log\log k)^2}\right)\right)$ \text{for large $k$}\\
\hline
\end{tabular}
\caption{Summary of the cases resolved to be lonely runner instances}
\label{tab:casesresolved}
\end{table}
\subsection{Preliminaries}\label{sec:preliminaries}
In this work, $\mathbb{N}$, $\mathbb{N}^+$, $\mathbb{Z}$, $\mathbb{R}$ and $\R^+$ represent the set of all non-negative integers, positive integers, integers, real numbers and positive real numbers respectively. Their analogues in $n$-dimensions are $\mathbb{N}^n$, $(\mathbb{N}^+)^n$, $\mathbb{Z}^n$, $\mathbb{R}^n$ and $(\R^+)^n$. Additionally, $\mathbf{e}_k \in \R^n$ denotes the $k$\textsuperscript{th} standard basis vector for $\R^n$. In particular, $\mathbf{e}_k = (\ldots, 0, 1, 0, 0, \ldots)$ where $k$\textsuperscript{th} entry is 1 and the rest are all zeroes. Furthermore, $\mathbf{e}=(1,\ldots, 1)$ represents the vector with all entries equal to $1$. We also denote the index set $\{1,2,\ldots, n\}$ by $[n]$.

Everytime the modulo operation is used, with modulus $m$,  we assume that the equivalence class is $[0,m)$. This is equal to the remainder on division by $m$. With the above, the fractional-part operation, $\{\cdot\}:\R^+\cup \{0\}\mapsto [0,1)$, is defined as the remainder on division by $1$. In particular, $\{a\}$ denotes the fraction part of $a$.

We now present a result on convex sets in $\R^2$ that will be crucial in proving one of our results. 
\begin{lem}{\label{lem:convexity}}
    Consider the convex sets ${S}_1, {S}_2 \in \R^2$ and the strip $\set{\x \in \R^2}{l \leq x_2 \leq u}$. If
    $$
        {S}_1 \cap {S}_2 \cap \set{\x \in \R^2}{x_2 = a} \neq \emptyset~~~~\forall~~a\in [l,u]
    $$
    then $({S}_1 \cup {S}_2) \cap \set{\x \in \R^2}{l \leq x_2 \leq u}$ is a convex set.
\end{lem}
\begin{proof}
Let ${U} = ({S}_1 \cup {S}_2) \cap \set{\x \in \R^2}{l \leq x_2 \leq u}$. Assume that ${U}$ is non-convex. Thus, there exists points $\x, \y \in {U}$ \ignore{Indeed, if both $\x^1$ and $\x^2$ belong to ${S}_1$ then the line segment joining $x_1$ and $x_2$ will lie entirely in ${S}_1$, since ${S}_1$ is convex. A similar argument can be made for ${S}_2$ as well. Assume now, without loss of generality, $\x^1 \in {S}_1, \x^2 \in {S}_2$.} such that a point in their convex hull (line segment joining $\x, \y$), is not contained in ${U}$. In particular, $\exists~~ \tilde{\x} = (\tilde{x},a) \in \conv(\x, \y)$ and $\tilde{\x} \not\in {U}$. 

Since $\x, \y \in {U}$ and $\tilde{\x} = (\tilde{x},a) \in \conv(\x, \y)$, it follows that $a \in [l,u]$. Furthermore, consider three additional points $(x_{{S}_1},a) \in {S}_1\setminus{S}_2$, $(x_{{S}_2},a) \in {S}_2\setminus{S}_1$ and $(x_{\cap},a)\in {S}_1\cap{S}_2$ such that either $x_{{S}_1} < \tilde{x} < x_{\cap}$ or $x_{\cap} < \tilde{x} < x_{{S}_2}$ (these points will exist because of our assumptions). If $x_{{S}_1} < \tilde{x} < x_{\cap}$, then since both $(x_{{S}_1},a), (x_{\cap},a)\in {S}_1$, their convex combination $(\tilde{x},a)\in {S}_1$ (since ${S}_1$ is convex), and consequently $(\tilde{x},a)\in{U}$, which poses a contradiction. Alternatively, if $x_{\cap} < \tilde{x} < x_{{S}_2}$, then since both $(x_{\cap},a), (x_{{S}_2},a)\in {S}_2$, their convex combination $(\tilde{x},a)\in {S}_2$ (since ${S}_2$ is convex), and consequently $(\tilde{x},a)\in{U}$, which also poses a contradiction. The result follows.
\end{proof}
\begin{figure}[ht]
\centering
\begin{tikzpicture}[scale=0.75]
    \draw[convexset,scale = 1.33]
            (2,-2) .. controls + (0.2,2) and + (-0.2,1) ..
            (4,-1) .. controls + (0.2,-0.5) and + (0.5,0) ..
            (3.2,-5) .. controls + (-0.5,0) and + (-0.1,-1) ..
            (2,-2) -- cycle;
            \node  at (4.2,-6) {\footnotesize${S}_2$};
    \draw[convexset,fill opacity=0.5]
            (0,0) .. controls + (0.2,0) and +(-1,0) ..
            (2,-4) .. controls + (1,0) and + (-0.2,0) ..
            (4,0) .. controls + (0.2,0) and + (1,0) ..
            (2,-7) .. controls + (-1,0) and + (-0.2,0) ..
            (0,0) -- cycle;
    \draw[segment,dashed] (-0.5,-1.5){} -- (5.5,-1.5){};
    \draw[segment,dashed] (-0.5,-5.5){} -- (5.5,-5.5){};
    \node  at (-1.3,-1.5) {\footnotesize$x_2=u$};
    \node  at (-1.3,-5.5) {\footnotesize$x_2=l$};
    \node  at (-1.3,-3) {\footnotesize$x_2=a$};
    \node  at (2,-6) {\footnotesize${S}_1$};
    \node  at (1,-1.8) {\footnotesize$\x$};
    \node  at (4,-5) {\footnotesize$\y$};
    \draw (0.5,-2) node {\tiny\textbullet};
    \draw (3.6,-5) node {\tiny\textbullet};
    \draw[dashed] (0.5,-2)--(3.6,-5);
    \draw (1.54,-3) node {\tiny\textbullet};
    \draw (0.75,-3) node {\tiny\textbullet};
    \draw (4.5,-3) node {\tiny\textbullet};
    \draw (3.3,-3) node {\tiny\textbullet};
    \draw[dashed] (-0.5,-3)--(5.5,-3);
    \node  at (1.6,-2.5) {\tiny$(\tilde{x},a)$};
    \node  at (0.75,-3.3) {\tiny$(x_{{S}_1},a)$};
    \node  at (3.13,-3.3) {\tiny$(x_{\cap},a)$};
    \node  at (4.35,-3.3) {\tiny$(x_{{S}_2},a)$};
\end{tikzpicture}
\caption{\small Depicting contradiction for Lemma \ref{lem:convexity} when either ${S}_1$ or ${S}_2$ is non-convex.}
\end{figure}
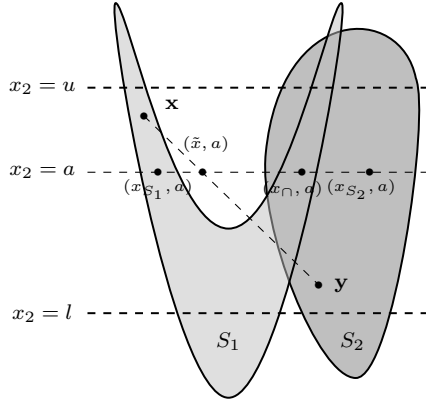

Since it suffices to prove the conjecture for integral speed vectors \cite{henze2017covering}, we will restrict our attention to the speed vectors $\n=(n_1,n_2,\ldots,n_k)\in(\mathbb{N}^+)^k$. We shall assume throughout that $n_i\geq n_j$ $\forall~~1\leq i < j \leq k$.

Consider a vector of speeds $\n\in(\mathbb{N}^+)^k$. Then, using the view-obstruction version of the conjecture, the \emph{Lonely Runner polyhedron}, ${P}(\n)$ defined by~\citet{polyhedra2019matthias} is
\begin{equation*}
{P}(\n):=\left\{\textbf{x}\in \mathbb{R}^k:\dfrac{n_i-kn_j}{k+1}\leq n_j x_i - n_i x_j\leq \dfrac{kn_i-n_j}{k+1}~\text{for}~1\leq i<j\leq k\right\} \ .
\end{equation*}
With respect to ${P}(\n)$ and any projection thereof considered in this work, we define the width of the set along a direction in the following.
\begin{dfn}
Consider ${S}\subseteq \mathbb{R}^n$ and $\mathbf{a}\in\mathbb{R}^n$. The width of ${S}$ in the direction of $\mathbf{a}$ is defined as $w_S(\mathbf{a}) = \underset{\x\in {S}}{\max}~\mathbf{a}\cdot \x-\underset{\x\in {S}}{\min}~\mathbf{a}\cdot \x$.
\end{dfn}
\ignore{
\begin{dfn}
The projection of $\mathbf{x} =(x_1,x_2,...,x_n)\in\mathbb{R}^n$ on to $\mathbb{R}^m$ is $\mathbf{x'}=(x_1,x_2,...,x_m)$ when $m\leq n$.
\end{dfn}

\begin{dfn}
$\mathbb{Z}^k$ is called the standard integer lattice. Any point $\mathbf{x}\in\mathbb{Z}^k$ is called an integer point.
\end{dfn}

\begin{dfn}
The translate of $\mathbf{x} =(x_1,x_2,...,x_n)\in\mathbb{R}^n$ by $a_i$ units, along the positive direction of each coordinate, is the point $\mathbf{x'}=(x_1+a_1,x_2+a_2,...,x_n+a_n)$.
\end{dfn}

\begin{dfn}
A translate is said to be an integer translate if all the $a_i$ above are integers.
\end{dfn}

\begin{dfn}
The translate of a set $X\in\mathbb{R}^n$ by $a_i$ units, along the positive direction of each coordinate, is $X'=\{\mathbf{x'}\in\mathbb{R}^n:\mathbf{x'}=(x_1+a_1,x_2+a_2,...,x_n+a_n)\}$.
\end{dfn}

For simplicity of notation, the above will be written as $X'=X+(a_1,a_2,...,a_n)$.
}
\begin{prop}\text{{\normalfont \cite{polyhedra2019matthias}}}
$\n\in(\mathbb{N}^+)^k$ is a lonely runner instance if and only if ${P}(\n)\cap\mathbb{Z}^k\neq \emptyset$.
\end{prop}
The intuition for the above is that, if an integer point $\mathbf{m}=(m_1,\ldots,m_k)$ lies in ${P}(\n)$, then there is a time $t$ when the $i^{\text{th}}$ runner has completed exactly $m_i$ rounds and all the runners are in the region $\left[\dfrac{1}{k+1},\dfrac{k}{k+1}\right]$.
\begin{dfn}
A time $t$ is a \emph{suitable time} for $\n$ if all runners are in the region $\left[\dfrac{1}{k+1},\dfrac{k}{k+1}\right]$ 
\ignore{$\n$ is a lonely runner instance} 
at time $t$.
\end{dfn}

\section{Our contributions}
Apart from the families of \emph{lonely runner instances} proposed in \cite{polyhedra2019matthias}, we propose two new families of instances satisfying \Cref{conj:lonelyrunner2}. We highlight these families of \emph{lonely runner instances} in Theorem \ref{thm:1} and Theorem \ref{thm:2}. The proofs of these two results are in Section~\ref{sec:mainproofs}. Polyhedral properties of suitable projections of ${P}(\n)$ are crucial ingredients in the proofs of the main results and are illustrated in Section \ref{sec:proofs}.
\begin{thm}\label{thm:1}
If $\n$ satisfies $k\geq 4$ and $n_2\left(\dfrac{k}{n_3}-\dfrac{1}{n_k}\right)\geq k+1$, then it is a lonely runner instance.
\end{thm}
\begin{rem}
 Examples of speed vectors that are \emph{lonely runner instances} due to Theorem \ref{thm:1} and not due to previous known results are: 
 $$(17,16,7,6,5,4,2), (18,16,7,6,5,4,3,2) \text{ and } (20,18,8,7,6,5,4,3,2).$$
\end{rem}
\begin{thm}\label{thm:2}
If $n_2\leq kn_k$ and $n_1~\textit{{\normalfont mod}}((k+1)n_k)\in [n_k,kn_k]$, then $\n$ is a lonely runner instance.
\end{thm}
\begin{rem}
Examples of vectors of speeds that are \emph{lonely runner instances} due to Theorem \ref{thm:2} and not due to previous known results are: 
$$(20,14,8,6,5,4,2), (24,14,10,9,8,6,5,2) \text{ and } (23,18,15,10,8,7,6,4,2).$$
\end{rem}
Furthermore, we provide an alternate proof for establishing the set of speeds satisfying $n_1\leq kn_k$ as \emph{lonely runner instances} \cite{polyhedra2019matthias}. In particular, we provide an explicit time at which these speed vectors become \emph{lonely runner instances}.
\begin{thm}\label{thm:6}
$t=\dfrac{k}{(k+1)n_1}$ is a suitable time for $\n$ if and only if $\n$ satisfies $n_1\leq kn_k$.
\end{thm}
It is known that $\n$ is a \emph{lonely runner instance} if and only if there is a suitable time $t$ that is at most $1$ \cite{bienia1998flows,tao2018some,perarnau2016correlation}. Intuitively, this makes sense because, at exactly unit time, each of the runners would be back at the start position. Thus, the combined motion of all the runners is periodic, with periodicity $1$.  We strengthen the condition on $t$ and show the following.
\begin{thm}\label{thm:7}
$\n$ is a lonely runner instance if and only if there is a suitable time $t\leq \dfrac{1}{2}$.
\end{thm}
We prove Theorems~\ref{thm:6} and \ref{thm:7} in Section~\ref{sec:timeproofs}. As a continuation to the characterization of families of \emph{lonely runner instances} through characterization of respective suitable times, as in Theorem \ref{thm:6} and Theorem \ref{thm:7}, we propose the following.
\begin{conj}\label{conj:1}
    For any $\n$ with $gcd(\n)=1$, there is always a suitable time of the form
    \begin{align*}
        \frac{m}{2^{\lceil\ln_{2}{(n_1)}+1\rceil}(k+1)n_1}
    \end{align*}
for some natural number $m$, where $\lceil{\cdot}\rceil$ denotes the ceiling function.
\end{conj}
\ignore{
\begin{conj}\label{conj:2}
    If $\n$ satisfies $n_i\leq kn_k$ and $n_{i-1}>kn_k$, then there is a suitable time of the form
    \begin{align*}
        \frac{m}{(k+1)\prod_{l=i}^{k} n_l}
    \end{align*}
for some natural number $m$ such that $n_1|m$ or $(n_1+n_k)|m$.
\end{conj}}
There are two reasons for why we believe Conjecture \ref{conj:1} to be true: First, considering $m=k 2^{\lceil\ln_{2}{(n_1)}+1\rceil}$ yields a suitable time for the family of vectors satisfying $n_1\leq kn_k$. Additionally, we have computationally verified the existence of a suitable time given by Conjecture~\ref{conj:1} for all possible $(\n,k)$ such that $n_1\leq 32$, and gcd($\n$) = 1. In particular, we have $2^{32}-1=4294967295$ speed vectors, $\n$, with $n_1\leq 32$. Of these, $4294900694$ are vectors with co-prime elements (satisfying $\textit{\normalfont gcd}(\n)=1$). Further among these, $2646877074~(\approx 61.62\%)$ different $\n$ are characterized \emph{lonely runner instances} due to the known results, including Theorems \ref{thm:1} and $\ref{thm:2}$. Conjecture \ref{conj:1} thus, if true, yields a characterization for the remaining $38.38\%$ of speed vectors. We further verified Conjecture \ref{conj:1} for $10000$ random speed vectors that satisfy $k\leq 50$ and $n_1\leq 1000$. All of the random speed vectors considered here were such that they were not characterized as \emph{lonely runner instances} by any of the earlier known results.

From a polyhedral theory perspective, the best known sufficiency condition for a speed vector $\n$ to be a \emph{lonely runner instance} is for the corresponding \emph{Lonely runner polyhedron}, $P(\n)$, to not be integer free. As a conclusion to this work, we provide a new sufficiency condition for a speed vector to be a \emph{lonely runner instance} in \Cref{sec:sufficiency}, \Cref{cor:sufficiency}.
\section{Some Properties of Lonely Runner Polyhedra}\label{sec:proofs}
The major construct in the proofs of Theorems~\ref{thm:1} and \ref{thm:2} is the \emph{Lonely Runner polyhedron}, ${P}(\n)$, as defined in Section \ref{sec:preliminaries}. In particular,
\begin{equation*}
{P}(\n):=\left\{\textbf{x}\in \mathbb{R}^k:\dfrac{n_i-kn_j}{k+1}\leq n_j x_i - n_i x_j\leq \dfrac{kn_i-n_j}{k+1}~\text{for}~1\leq i<j\leq k\right\}
\end{equation*}
\setcounter{thm}{10}
\begin{prop}\label{thm:symmetry}
$P(\n)$ is symmetric about $-\dfrac{1}{2}\mathbf{e}$.
\end{prop}
\begin{proof}
Let $\mathbf{v}=(v_1,...,v_k)$ be an arbitrary point in $\R^k$. With this,
\begin{align*}
    -\dfrac{1}{2}\mathbf{e}+\mathbf{v}\in P(\n) & \iff \dfrac{n_i-kn_j}{k+1}\leq n_j\left(-\dfrac{1}{2}+v_i\right)-n_i\left(-\dfrac{1}{2}+v_j\right)\leq \dfrac{kn_i-n_j}{k+1}~~\forall 1\leq i<j\leq k\\
    & \iff -\dfrac{(k-1)(n_i+n_j)}{2(k+1)}\leq n_j v_i-n_i v_j\leq \dfrac{(k-1)(n_i+n_j)}{2(k+1)}~~\forall 1\leq i<j\leq k\\
    & \iff -\dfrac{(k-1)(n_i+n_j)}{2(k+1)}\leq -n_j v_i+n_i v_j\leq \dfrac{(k-1)(n_i+n_j)}{2(k+1)}~~\forall 1\leq i<j\leq k\\
    & \iff \dfrac{n_i-kn_j}{k+1}\leq n_j\left(-\dfrac{1}{2}-v_i\right)-n_i\left(-\dfrac{1}{2}-v_j\right)\leq \dfrac{kn_i-n_j}{k+1}~~\forall 1\leq i<j\leq k\\
    & \iff -\dfrac{1}{2}\mathbf{e}-\mathbf{v}\in P(\n)
\end{align*}
The result follows.
\end{proof}

A very important fact that we use to prove our main results is that any integer translate of a polyhedron, say ${P}+\mathbf{v}$, $\mathbf{v}\in \Z^k$, contains the same number of integer points as the original polyhedron, ${P}$. To observe this, first note that $\mathbf{a}\cdot \mathbf{x}=b$ is a face of ${P}$ if and only if $\mathbf{a}\cdot (\mathbf{x}-\mathbf{v})=b$ is a face of ${P}+\mathbf{v}$. Then, $\mathbf{l}\in \mathbb{Z}^k$ satisfies $\mathbf{a}\cdot \mathbf{x}=b$ if and only if $\mathbf{l}+\mathbf{v}$ satisfies $\mathbf{a}\cdot (\mathbf{x}-\mathbf{v})=b$. Due to this bijection between ${P}$ and ${P}+\mathbf{v}$, the number of integer points in ${P}$ and ${P}+\mathbf{v}$ are equal.\\
We also utilize projection arguments to prove some of the results. Consider the projection of ${P}(\n)$ on to the first $m$ coordinates, ${P}_m(\n)$. If $\n$ satisfies $n_{m+1}\leq kn_k$, it suffices to show that the ${P}_m(\n)$ contains an integer point $p=(p_1,\ldots,p_m)$ instead of proving that ${P}(\n)$ is not integer lattice-free. This is because, $p'=(p_1,\ldots,p_m,0,\ldots,0)$ would lie in ${P}(\n)$.

Consider the projection of ${P}(\n)$ on to the first two coordinates, ${P}_2(\n)$. Indeed ${P}_2(\n)$ is a polyhedron as well \cite{ziegler2012lectures}. For notational brevity,  we shall henceforth refer to ${P}_2(\n)$ as ${Q}$. The generating inequalities of ${Q}$ are:
\begin{gather*}
    \dfrac{n_1}{(k+1)n_k}-\dfrac{k}{k+1}\leq x_1\leq \dfrac{kn_1}{(k+1)n_3}-\dfrac{1}{k+1}\\
    \dfrac{n_2}{(k+1)n_k}-\dfrac{k}{k+1}\leq x_2\leq \dfrac{kn_2}{(k+1)n_3}-\dfrac{1}{k+1}\\
    \dfrac{n_1-kn_2}{k+1}\leq n_2x_1-n_1x_2\leq \dfrac{kn_1-n_2}{k+1}
\end{gather*}
In addition to the generating inequalities of ${Q}$, consider the following lines (Figures \ref{fig:1} and \ref{fig:2})
\begin{align*}
    &l_1:n_2x_1-n_1x_2=\dfrac{kn_1-n_2}{k+1}&& l_2:n_2x_1-n_1x_2=\dfrac{n_1-kn_2}{k+1}\\
    &L_1:x_2= \alpha = \dfrac{n_2}{(k+1)n_k}+\dfrac{1}{k+1}&& L_2:x_2= \beta = \dfrac{n_2}{(k+1)n_k}+\dfrac{2n_2}{(k+1)n_1}-\dfrac{k}{k+1}\\
    &L_3:x_2= \gamma = \dfrac{kn_2}{(k+1)n_3}-\dfrac{2n_2}{(k+1)n_1}-\dfrac{1}{k+1}&& L_4:x_2=\delta=\dfrac{1}{k+1}\left(\dfrac{n_2}{n_k}-1\right)\\
    &L_5:x_2=\zeta = \dfrac{k}{k+1}\left(\dfrac{n_2}{n_3}-1\right)&&L_6:x_1=\kappa=\dfrac{n_1}{(k+1)n_k}+\dfrac{1}{k+1}\\
\end{align*}
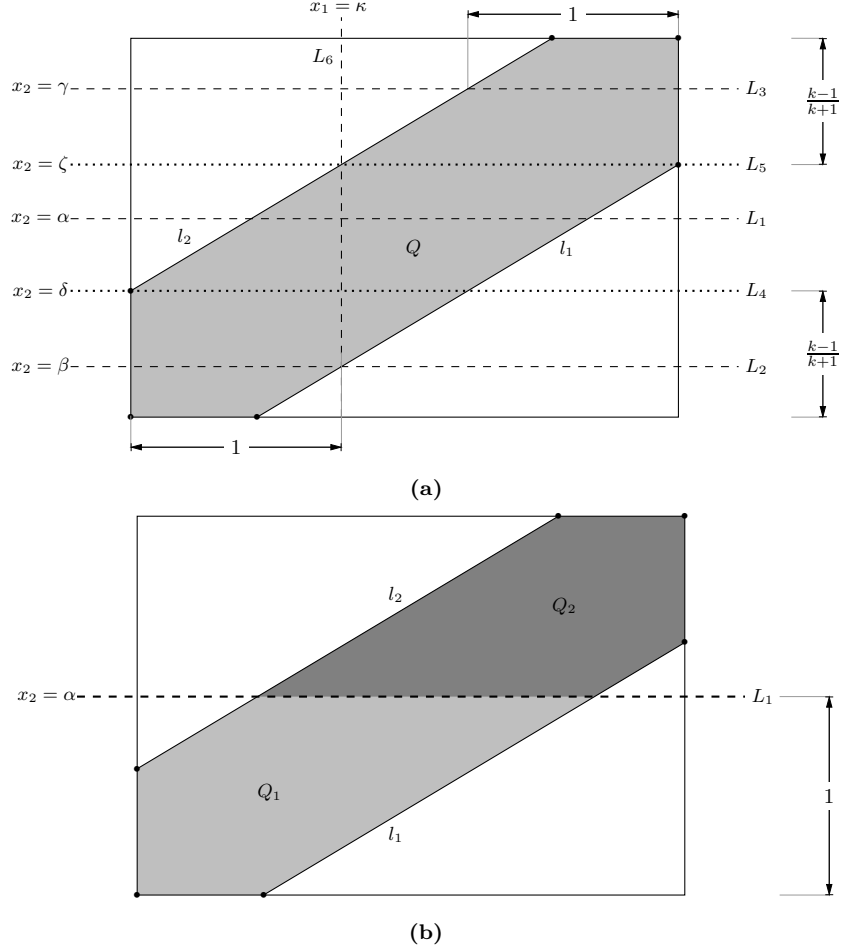
\begin{figure}
\centering
\begin{subfigure}[b]{\textwidth}
\centering
   \begin{tikzpicture}[scale=0.8,every node/.style={scale=0.8}]
        \coordinate (a) at (0,0);
        \coordinate (b) at (9.1,6.3);
        \coordinate (c) at (2.1,0);
        \coordinate (d) at (0,2.1);
        \coordinate (e) at (9.1,4.2);
        \coordinate (f) at (7,6.3);
        \coordinate (A) at (-1,0.84);
        \coordinate (B) at (10.1,0.84);
        \coordinate (C) at (-1,5.46);
        \coordinate (D) at (10.1,5.46);
        \coordinate (E) at (-1,2.1);
        \coordinate (F) at (10.1,2.1);
        \coordinate (G) at (-1,4.2);
        \coordinate (H) at (10.1,4.2);
        \coordinate (I) at (-1,3.5);
        \coordinate (J) at (10.1,3.5);
        \coordinate (K) at (2.1,0);
        \coordinate (L) at (2.1,6.3);
        \coordinate (M) at (3.5,0);
        \coordinate (N) at (3.5,6.65);
        \coordinate (O) at (-1,3.3);
        \coordinate (P) at (10.1,3.3);
        \fill[lightgray] (c)--(a)--(d)--(f)--(b)--(e)--cycle;
        \draw (a) node {\tiny\textbullet};
        \draw (b) node {\tiny\textbullet};
        \draw (c) node {\tiny\textbullet};
        \draw (d) node {\tiny\textbullet};
        \draw (e) node {\tiny\textbullet};
        \draw (f) node {\tiny\textbullet};
        \draw (c) to (e);
        \draw (d) to (f);
        \draw (0,0) rectangle (9.1,6.3);
        \draw[dashed] (A) to (B);
        \draw[dashed] (C) to (D);
        \draw[dotted,thick] (E) to (F);
        \draw[dotted,thick] (G) to (H);
        \draw[dashed] (M) to (N);
        \draw[dashed] (O) to (P);
        \node  at (7.25,2.8) {\small$l_1$};
        \node  at (0.9,3) {\small$l_2$};
        \node  at (10.4,5.46) {\small$L_3$};
        \node  at (10.4,3.3) {\small$L_1$};
        \node  at (10.4,0.84) {\small$L_2$};
        \node  at (10.4,4.2) {\small$L_5$};
        \node  at (10.4,2.1) {\small$L_4$};
        \node  at (-1.5,5.46) {\small$x_2=\gamma$};
        \node  at (-1.5,0.84) {\small$x_2 = \beta$};
        \node  at (-1.5,2.1) {\small$x_2 = \delta$};
        \node  at (-1.5,4.2) {\small$x_2 = \zeta$};
        \node  at (-1.5,3.3) {\small$x_2 = \alpha$};
        \node  at (3.45,6.8) {\small$x_1 = \kappa$};
        \node  at (4.7,2.8) {\small${Q}$};
        \node at (3.2,6){\small$L_6$};
        \dimline[line style = {line width=0.5},extension start length=-0.25, extension end length=-0.25,label style = {no slope}]{(11.5,0)}{(11.5,2.1)}{$\frac{k-1}{k+1}$};
        \dimline[line style = {line width=0.5},extension start length=0.25, extension end length=0.25,label style = {no slope}]{(11.5,6.3)}{(11.5,4.2)}{$\frac{k-1}{k+1}$};
        \dimline[line style = {line width=0.5},extension start length=-0.15, extension end length=-0.45,label style = {no slope}]{(0,-0.5)}{(3.5,-0.5)}{$1$};
        \dimline[line style = {line width=0.5},extension start length=0.35, extension end length=0.1,label style = {no slope}]{(5.6,6.7)}{(9.1,6.7)}{$1$};
    \end{tikzpicture}
   \caption{}
   \label{fig:1} 
\end{subfigure}
\begin{subfigure}[b]{\textwidth}
\centering
       \begin{tikzpicture}[scale=0.8,every node/.style={scale=0.8}]
        \coordinate (a) at (0,0);
        \coordinate (b) at (9.1,6.3);
        \coordinate (c) at (2.1,0);
        \coordinate (d) at (0,2.1);
        \coordinate (e) at (9.1,4.2);
        \coordinate (f) at (7,6.3);
        \coordinate (I) at (-1,3.3);
        \coordinate (J) at (10.1,3.3);
        \fill[lightgray] (c)--(a)--(d)--(f)--(b)--(e)--cycle;
        \fill[gray] (f)--(b)--(e)--(7.6,3.3)--(2,3.3)--cycle;
        \draw (a) node {\tiny\textbullet};
        \draw (b) node {\tiny\textbullet};
        \draw (c) node {\tiny\textbullet};
        \draw (d) node {\tiny\textbullet};
        \draw (e) node {\tiny\textbullet};
        \draw (f) node {\tiny\textbullet};
        \draw (c) to (e);
        \draw (d) to (f);
        \draw (0,0) rectangle (9.1,6.3);
        \draw[dashed,thick] (I) to (J);
        \node  at (4.3,5) {\small$l_2$};
        \node  at (4.3,1) {\small$l_1$};
        \node  at (10.4,3.3) {\small$L_1$};
        \node  at (-1.5,3.3) {\small$x_2 = \alpha$};
        \node  at (2.2,1.7) {\small${Q}_1$};
        \node  at (7.1,4.8) {\small${Q}_2$};
        \dimline[line style = {line width=0.5},extension start length=-0.25, extension end length=-0.25,label style = {no slope}]{(11.5,0)}{(11.5,3.3)}{$1$};
    \end{tikzpicture}
   \caption{}
   \label{fig:2}
\end{subfigure}
\caption[]{\small Polyhedron $Q$ under the assumption $n_2\left(\dfrac{k}{n_3}-\dfrac{1}{n_k}\right)\geq k+1$}
\end{figure}
\ignore{
\begin{align*}
&l_1:n_2x_1-n_1x_2=\dfrac{kn_1-n_2}{k+1}& \text{(The lower slant edge of ${Q}$)}\\
&l_2:n_2x_1-n_1x_2=\dfrac{n_1-kn_2}{k+1}& \text{(The upper slant edge of ${Q}$)}\\
&L_1:x_2=\dfrac{n_2}{(k+1)n_k}+\dfrac{1}{k+1}& \text{(Line that is one unit above the bottom of ${Q}$)}\\
&L_2:x_2=\dfrac{n_2}{(k+1)n_k}+\dfrac{2n_2}{(k+1)n_1}-\dfrac{k}{k+1}& \text{(Bottom most line where ${Q}$ is at least a unit long)}\\
&L_3:x_2=\dfrac{kn_2}{(k+1)n_3}-\dfrac{2n_2}{(k+1)n_1}-\dfrac{1}{k+1}& \text{(Top most line where ${Q}$ is at most a unit long)}\\
&L_4:x_2=\dfrac{1}{k+1}\left(\dfrac{n_2}{n_k}-1\right)& \text{\Big{(}Line that is $\dfrac{k-1}{k+1}$ units above the bottom of ${Q}$\Big{)}}\\
&L_5:x_2=\dfrac{k}{k+1}\left(\dfrac{n_2}{n_3}-1\right)& \text{\Big{(}Line that is $\dfrac{k-1}{k+1}$ units below the top of ${Q}$\Big{)}}\\
&L_6:x_1=\dfrac{n_1}{(k+1)n_k}+\dfrac{1}{k+1}&\text{(Line that is a unit right of the left edge of ${Q}$)}\\
&L_7:x_1=\dfrac{n_1}{(k+1)n_k}& \text{\Big{(}Line that is $\dfrac{k-1}{k+1}$ units right of the left edge of ${Q}$\Big{)}}
\end{align*}}
\begin{lem}\label{lem:lengthandwidth}
If $\n$ satisfies $n_2\left(\dfrac{k}{n_3}-\dfrac{1}{n_k}\right)\geq k+1$, then $w_{Q}(\mathbf{e}_1) \geq 1$ and $w_{Q}(\mathbf{e}_2) \geq 1$.
\end{lem}
\begin{proof}
The lines $x_2=\zeta+\dfrac{k-1}{k+1}$ and $x_2=\delta-\dfrac{k-1}{k+1}$ represent diametrically opposite facets (top and bottom edges respectively) of ${Q}$. Then, $w_{Q}(\mathbf{e}_2)$ is the distance between these facets of ${Q}$.
\begin{align}
    w_{Q}(\mathbf{e}_2) & = \left(\dfrac{kn_2}{(k+1)n_3}-\dfrac{1}{k+1}\right)-\left(\dfrac{n_2}{(k+1)n_k}-\dfrac{k}{k+1}\right)\nonumber\\
      & =\dfrac{n_2}{k+1}\left(\dfrac{k}{n_3}-\dfrac{1}{n_k}\right)+\dfrac{k-1}{k+1}\label{ineq:wQe2}\\
      & \geq \dfrac{k+1}{k+1}+\dfrac{k-1}{k+1}\nonumber\\
      & \geq 1\nonumber
\end{align}
where the first inequality follows from our assumption and the latter follows from $k\geq 1$. The lines $x_1=\dfrac{kn_1}{(k+1)n_3}-\dfrac{1}{k+1}$ and $x_1=\dfrac{n_1}{(k+1)n_k}-\dfrac{k}{k+1}$ represent diametrically opposite facets (right and left edges respectively) of ${Q}$. $w_{Q}(\mathbf{e}_1)$ is the distance between these facets and it can be expressed as
\begin{align*}
    w_{Q}(\mathbf{e}_1) & = \left(\dfrac{kn_1}{(k+1)n_3}-\dfrac{1}{k+1}\right)-\left(\dfrac{n_1}{(k+1)n_k}-\dfrac{k}{k+1}\right)\\
      & = \dfrac{n_1}{k+1}\left(\dfrac{k}{n_3}-\dfrac{1}{n_k}\right)+\dfrac{k-1}{k+1}\\
      & \geq \dfrac{n_2}{k+1}\left(\dfrac{k}{n_3}-\dfrac{1}{n_k}\right)+\dfrac{k-1}{k+1}\\
      & = w_{Q}(\mathbf{e}_2)\\
      & \geq 1
\end{align*}
where the first inequality follows from $n_1\geq n_2$, the following equality from \eqref{ineq:wQe2} and the last inequality from $w_{Q}(\mathbf{e}_2)\geq 1$. The result follows.
\end{proof}
\Cref{lem:lengthandwidth} suggests that ${Q}_1 \define {Q}\cap\left\{\mathbf{x}\in\mathbb{R}^2:x_2\leq \alpha\right\} \neq \emptyset$ and ${Q}_2 \define {Q}\cap\left\{\mathbf{x}\in\mathbb{R}^2:x_2\geq \alpha\right\} \neq \emptyset$. Consider ${Q}_3:={Q}_2-\mathbf{e}_2$ and ${Q}_4:={Q}_1\cup{Q}_3$. Observe that, by definition, $w_{{Q}_1}(\mathbf{e}_2) = 1$ (\Cref{fig:2}).
\begin{lem}\label{lem:width}
If $\n$ satisfies $n_2\left(\dfrac{k}{n_3}-\dfrac{1}{n_k}\right)\geq k+1$, then $w_{{Q}_3}(\mathbf{e}_2) \geq \dfrac{k-1}{k+1}$.
\end{lem}
\begin{proof}
${Q}_3$ is a translate of ${Q}_2$\ignore{ and translation does not change the width of a polyhedron}. Thus, it suffices to show that $w_{{Q}_2}(\mathbf{e}_2) \geq \dfrac{k-1}{k+1}$\ignore{, which is the distance between $L_1$ and $x_2=\dfrac{kn_2}{(k+1)n_3}-\dfrac{1}{k+1}$, is at least $\dfrac{k-1}{k+1}$ units}. We have,
\begin{align*}
    w_{{Q}_{2}}(\mathbf{e}_2) & = \left(\zeta + \dfrac{k-1}{k+1}\right) - \alpha\\ 
    & = \left(\dfrac{kn_2}{(k+1)n_3}-\dfrac{1}{k+1}\right)-\left(\dfrac{n_2}{(k+1)n_k}+\dfrac{1}{k+1}\right)\\
      & = \dfrac{n_2}{k+1}\left(\dfrac{k}{n_3}-\dfrac{1}{n_k}\right)-\dfrac{2}{k+1}\\
      & \geq \dfrac{k+1}{k+1}-\dfrac{2}{k+1}\\
      & =\dfrac{k-1}{k+1}
\end{align*}
where the inequality follows from our assumption. The result follows.
\end{proof}
\begin{lem}\label{lem:line}
If $\n$ satisfies $2n_1\leq (k-1)n_2$ and $n_2\left(\dfrac{k}{n_3}-\dfrac{1}{n_k}\right)\geq k+1$, then
$$
    \{\mathbf{x}\in\mathbb{R}^2:x_2=a\}\cap{Q}_1\cap{Q}_3\neq\emptyset\hspace{5mm} \forall~~a\in[\alpha-1,\delta].
$$
\end{lem}
\begin{proof}
Consider $a\in[\alpha-1,\delta]$, and the line $x_2=a$. From \Cref{lem:width} and the fact that $w_{{Q}_{1}}(\mathbf{e}_2)=1$, we have $\{\mathbf{x}\in\mathbb{R}^2:x_2=a\}\cap{Q}_1\neq\emptyset$ and $\{\mathbf{x}\in\mathbb{R}^2:x_2=a\}\cap{Q}_3\neq\emptyset$.

Since $\alpha-1\leq a\leq \delta<\zeta$, the lines $x_2=a$ and $l_1$ will intersect. Let $A$ be the point of intersection of these lines. Specifically,
$$
A=\left(\dfrac{kn_1-n_2}{(k+1)n_2}+\dfrac{n_1a}{n_2},a\right)
$$
By definition, $A\in {Q}_1$. Additionally, any point on the line $x_2 = a$, such that $\kappa-1\leq x_1 \leq \dfrac{kn_1-n_2}{(k+1)n_2}+\dfrac{n_1a}{n_2}$ will be in ${Q}_1$ as well.

Let $B$ be the point of intersection of $x_2=a+1$ with $l_2$. Observe that $B \in {Q}_2$. In particular,
$$
B=\left(\dfrac{n_1-kn_2}{(k+1)n_2}+\dfrac{n_1(a+1)}{n_2},a+1\right)
$$
Since ${Q}_3 = {Q}_2-\mathbf{e}_2$, $\exists~~ B' \in {Q}_3$ such that $B' = B - \mathbf{e}_2$. In particular,
$$
B'=\left(\dfrac{n_1-kn_2}{(k+1)n_2}+\dfrac{n_1(a+1)}{n_2},a\right)
$$
As noted earlier, to prove $B' \in {Q}_1$ it suffices to show that 
$$
\kappa -1 \leq \dfrac{n_1-kn_2}{(k+1)n_2}+\dfrac{n_1(a+1)}{n_2} \leq \dfrac{kn_1-n_2}{(k+1)n_2}+\dfrac{n_1a}{n_2}.
$$
We have,
\begin{align}
    \left(\dfrac{n_1-kn_2}{(k+1)n_2}+\dfrac{n_1(a+1)}{n_2}\right)-\left(\dfrac{kn_1-n_2}{(k+1)n_2}+\dfrac{n_1a}{n_2}\right) & = \dfrac{n_1}{n_2}\left(\dfrac{2}{k+1}\right)-\dfrac{k-1}{k+1}\nonumber\\
    & = \dfrac{1}{k+1}\left(\dfrac{2n_1}{n_2}-(k-1)\right)\nonumber\\
    & \leq 0 \label{ineq:1}
\end{align}
where the last inequality follows from our assumption. Additionally, note that
\begin{align}
    \left(\dfrac{n_1-kn_2}{(k+1)n_2}+\dfrac{n_1(a+1)}{n_2}\right)-(\kappa-1) & = \left(\dfrac{n_1-kn_2}{(k+1)n_2}+\dfrac{n_1(a+1)}{n_2}\right)-\left(\dfrac{n_1}{(k+1)n_k}-\dfrac{k}{k+1}\right)\nonumber\\
    & \geq \left(\dfrac{n_1-kn_2}{(k+1)n_2}+\dfrac{n_1\alpha}{n_2}\right)-\left(\dfrac{n_1}{(k+1)n_k}-\dfrac{k}{k+1}\right)\nonumber\\
    & = \left(\dfrac{n_1}{(k+1)n_2}+\dfrac{n_1}{(k+1)n_k} + \dfrac{n_1}{(k+1)n_2}\right)-\dfrac{n_1}{(k+1)n_k}\nonumber\\
    & = \dfrac{2n_1}{(k+1)n_2}\nonumber\\
    & > 0 \label{ineq:2}
\end{align}
where the second inequality is immediate from $\alpha -1 \leq a$ and the last inequality follows from the positivity of $n_1, n_2$ and $k$. From \eqref{ineq:1} and \eqref{ineq:2} it follows that $B' \in {Q}_1$. Combining with $B' \in \{\mathbf{x}\in\mathbb{R}^2:x_2=a\}$ and $B' \in {Q}_3$ yields the result.
\ignore{
Note that the points $A$ and $B'$ lie on the line $x_2=a$ and $A$ lies in ${Q}_1$. Since $d\leq 0$, we have $B' \in {Q}_1$. Therefore, $\{\mathbf{x}\in\mathbb{R}^2:x_2=a\}\cap{Q}_1\cap{Q}_3=\overline{B'A}\neq\emptyset$.}
\end{proof}

\begin{lem}\label{lem:length}
If $\n$ satisfies $k\geq 3$, $2n_1\leq (k-1)n_2$ and $n_2\left(\dfrac{k}{n_3}-\dfrac{1}{n_k}\right)\geq k+1$, then 
$$
w_{\set{\x\in \R^2}{x_2 = a} \cap {Q}_4}(\mathbf{e}_1) \geq 1~~~~~~\forall~~a \in[\alpha -1,\alpha].
$$
\end{lem}
\begin{proof}
First, consider $a \in[\delta,\alpha]$. Then, $w_{\set{\x\in \R^2}{x_2 = a} \cap {Q}_4}(\mathbf{e}_1)$ is at least as much as the distance between $l_1$ and $l_2$ in the $x_1$-direction.
\begin{align}
    w_{\set{\x\in \R^2}{x_2 = a} \cap {Q}_4}(\mathbf{e}_1) & = \dfrac{1}{n_2}\left(\dfrac{kn_1-n_2}{k+1}-\dfrac{n_1-kn_2}{k+1}\right)\nonumber\\
      & = \dfrac{k-1}{k+1}\dfrac{n_1+n_2}{n_2}\nonumber\\
      & \geq  2\dfrac{k-1}{k+1}\nonumber\\
      & \geq 1.\label{ineq-lem:13-1}
\end{align}
where the first inequality follows from $n_1\geq n_2$ while the latter follows from $k\geq 3$.

Now, consider $a\in[\alpha-1,\delta]$. Since ${Q}_1$ and ${Q}_3$ are convex, as a result of \Cref{lem:line} and \Cref{lem:convexity}, ${Q}_4\cap\{\mathbf{x}\in\mathbb{R}^2:\alpha-1\leq x_2\leq\delta\}$ is convex.

Since $l_1$ is a facet of ${Q}_2$ and ${Q}_3={Q}_2-\mathbf{e}_2$, it follows that $l_1':=l_1-\mathbf{e}_2$ is a facet of ${Q}_3$. In particular, we have $l_1':n_2x_1-n_1(x_2+1)=\dfrac{kn_1-n_2}{k+1}$.

Let $C$ be the point of intersection of $l_1'$ and $x_2=\alpha-1$. Then:
\begin{center}
    $C=\left(\dfrac{n_1}{n_2}+\dfrac{n_1}{(k+1)n_k}-\dfrac{1}{k+1},\dfrac{n_2}{(k+1)n_k}-\dfrac{k}{k+1}\right)$
\end{center}
It follows that $\set{\x\in \R^2}{x_2 = \alpha-1} \cap {Q}_4 = [(\kappa-1,\alpha-1), C]$. We, thus, have
\begin{align}
    w_{\set{\x\in \R^2}{x_2 = \alpha-1} \cap {Q}_4}(\mathbf{e}_1) & = \left(\dfrac{n_1}{n_2}+\dfrac{n_1}{(k+1)n_k}-\dfrac{1}{k+1}\right)-\left(\dfrac{n_1}{(k+1)n_k}-\dfrac{k}{k+1}\right)\nonumber\\
      & = \dfrac{n_1}{n_2}+\dfrac{k-1}{k+1}\nonumber\\
      & > 1.\label{ineq-lem:13-2}
\end{align}
where the inequality follows from $n_1\geq n_2$ and $\dfrac{k-1}{k+1}>0$.

In case $\zeta-\alpha\geq\dfrac{k-1}{k+1}$, $l_1'$ would intersect $x_2=\delta$. This point, say $D$, would be given by:
\begin{center}
    $D=\left(\dfrac{2kn_1}{(k+1)n_2}+\dfrac{n_1}{(k+1)n_k}-\dfrac{1}{k+1},\dfrac{1}{k+1}\left(\dfrac{n_2}{n_k}-1\right)\right)$
\end{center}
Consequently, $\set{\x\in \R^2}{x_2 = \delta} \cap {Q}_4 = [(\kappa-1,\delta), D]$. We, then, have:
\begin{align}
    w_{\set{\x\in \R^2}{x_2 = \delta} \cap {Q}_4}(\mathbf{e}_1) & = \left(\dfrac{2kn_1}{(k+1)n_2}+\dfrac{n_1}{(k+1)n_k}-\dfrac{1}{k+1}\right)-\left(\dfrac{n_1}{(k+1)n_k}-\dfrac{k}{k+1}\right)\nonumber\\
      & = \dfrac{2kn_1}{(k+1)n_2}+\dfrac{k-1}{k+1}\nonumber\\
      & > 1.\label{ineq-lem:13-3}
\end{align}
where the inequality follows from $n_1\geq n_2$ and $k\geq 3$. If $\zeta<\alpha+\dfrac{k-1}{k+1}$, then $l_1'$ would not intersect $x_2=\delta$. Consequently, we have 
$$
\set{\x\in \R^2}{x_2 = \delta} \cap {Q}_4=\left[(\kappa-1,\delta),\left(\dfrac{kn_1}{(k+1)n_3}-\dfrac{1}{k+1},\delta\right)\right].
$$
Combining this with \Cref{lem:lengthandwidth}, we have:
\begin{align}
    w_{\set{\x\in \R^2}{x_2 = \delta} \cap {Q}_4}(\mathbf{e}_1)=w_{{Q}_4}(\mathbf{e}_1)=w_{{Q}}(\mathbf{e}_1)>1\label{ineq-lem:13-4}
\end{align}
Since ${Q}_4\cap\{\mathbf{x}\in\mathbb{R}^2:\alpha-1\leq x_2\leq\delta\}$ is convex, it follows from \eqref{ineq-lem:13-2}, \eqref{ineq-lem:13-3} and \eqref{ineq-lem:13-4} that
\begin{align}
    w_{\set{\x\in \R^2}{x_2 = a} \cap {Q}_4}(\mathbf{e}_1)>1~~\forall a\in[\alpha-1,\delta]\label{ineq-lem:13-5}
\end{align}
\ignore{Now, consider $a\in\left[\dfrac{n_2}{(k+1)n_k}-\dfrac{k}{k+1},\dfrac{1}{k+1}\left(\dfrac{n_2}{n_k}-1\right)\right]$.
Now, consider $a\in [\alpha-1, \delta]$. Due to \cref{lem:line}, $\{\mathbf{x}\in\mathbb{R}^2:x_2=a\}\cap{Q}_1\cap{Q}_3\neq\emptyset~~\forall a\in [\alpha-1, \delta]$. For all $a$ in the above interval, ${Q}_4$ has a common facet $x_1=\dfrac{x_1}{(k+1)n_k}-\dfrac{k}{k+1}$. The translate of $l_1$ has a positive slope and the facet $x_1=\dfrac{kn_1}{(k+1)n_3}-\dfrac{1}{k+1}$ is parallel to the $x_2$-axis. Due to this, 
$$
w_{\set{\x\in \R^2}{x_2 = \alpha-1} \cap {Q}_4}(\mathbf{e}_1)\leq w_{\set{\x\in \R^2}{x_2 = a} \cap {Q}_4}(\mathbf{e}_1)~\forall a\in[\alpha-1,\delta].
$$ 
With this, it suffices to show that $w_{\set{\x\in \R^2}{x_2 = \alpha-1} \cap {Q}_4}(\mathbf{e}_1)\geq 1$.

Let $C$ be the point of intersection of $l_1$ and $L_1$. Upon translation of ${Q}_2$, $C$ would be translated to $C'$, given by:
$$
C'=\left(\dfrac{n_1}{n_2}+\dfrac{n_1}{(k+1)n_k}-\dfrac{1}{k+1},\dfrac{n_2}{(k+1)n_k}-\dfrac{k}{k+1}\right)
$$
Using this, we have:
\begin{align*}
    w_{\set{\x\in \R^2}{x_2 = \alpha-1} \cap {Q}_4}(\mathbf{e}_1) & = \left(\dfrac{n_1}{n_2}+\dfrac{n_1}{(k+1)n_k}-\dfrac{1}{k+1}\right)-\left(\dfrac{n_1}{(k+1)n_k}-\dfrac{k}{k+1}\right)\\
    & = \dfrac{n_1}{n_2} + \dfrac{k-1}{k+1}\\
    & > 1
\end{align*}}

The result follows from \eqref{ineq-lem:13-1} and \eqref{ineq-lem:13-5}. 
\end{proof}
\begin{lem}\label{lem:lines}
Define ${Q}_5 \define {Q}\cap\set{\x\in\R^2}{\beta \leq x_2 \leq \gamma}.$ If $\n$ satisfies $k\geq4$, $2n_1>(k-1)n_2$ and $n_2\left(\dfrac{k}{n_3}-\dfrac{1}{n_k}\right)\geq k+1$, then $w_{{Q}_5}(\mathbf{e}_2) > 1$.
\end{lem}
\begin{proof}
It is immediate from the definition of ${Q}_5$ that $w_{{Q}_5}(\mathbf{e}_2) = \gamma - \beta$.
\begin{alignat*}{2}
w_{{Q}_5}(\mathbf{e}_2) &=\left(\dfrac{kn_2}{(k+1)n_3}-\dfrac{2n_2}{(k+1)n_1}-\dfrac{1}{k+1}\right)-\left(\dfrac{n_2}{(k+1)n_k}+\dfrac{2n_2}{(k+1)n_1}-\dfrac{k}{k+1}\right)\\
    &=\dfrac{kn_2}{(k+1)n_3}-\dfrac{n_2}{(k+1)n_k}-\dfrac{4n_2}{(k+1)n_1}+\dfrac{k-1}{k+1}\\
    &>\dfrac{kn_2}{(k+1)n_3}-\dfrac{n_2}{(k+1)n_k}-\dfrac{8}{(k-1)(k+1)}+\dfrac{k-1}{k+1}\\
    &=\dfrac{kn_2}{(k+1)n_3}-\dfrac{n_2}{(k+1)n_k}-\dfrac{8}{(k-1)(k+1)}+1-\dfrac{2}{k+1}\\
    &=\dfrac{1}{k+1}\left(\dfrac{kn_2}{n_3}-\dfrac{n_2}{n_k}-2\right)+1-\dfrac{8}{(k-1)(k+1)}\\
    &\geq \dfrac{k-1}{k+1}+1-\dfrac{8}{(k-1)(k+1)}\\
    &=1+\dfrac{(k-1)^2-8}{(k-1)(k+1)}\\ 
    &> 1.
\end{alignat*}
where the first inequality follows from the assumption $2n_1>(k-1)n_2$, the second follows from $n_2\left(\dfrac{k}{n_3}-\dfrac{1}{n_k}\right)\geq k+1$ and the final follows from $k\geq 4$. The result follows.
\end{proof}

\section{Proofs of Theorems~\ref{thm:1} and \ref{thm:2}}\label{sec:mainproofs}
\setcounter{mainthm}{4}
\begin{mainthm}
$\n$ is a lonely runner instance if it satisfies $k\geq 4$ and  $n_2\left(\dfrac{k}{n_3}-\dfrac{1}{n_k}\right)\geq k+1$.
\end{mainthm}
\begin{proof}
Combining $n_2\left(\dfrac{k}{n_3}-\dfrac{1}{n_k}\right)\geq k+1$ with the positivity of $n_2, n_k$ and $k$ immediately yields that $n_3 < kn_k$.
\ignore{\begin{align*}
    n_3 & \leq kn_k\left(\dfrac{1}{1+\dfrac{n_k}{n_2}(k+1)}\right)\\
    & < kn_k.
\end{align*}}
Thus, it suffices to show that ${Q} = {P}_2(\n)$ is not integer lattice-free. In particular, if there exists an integer point in ${Q}$, then $\n$ is a \emph{lonely runner instance}.

Assume that $2n_1\leq (k-1)n_2$. By definition of $L_1$, we have $w_{{Q}_4}(\mathbf{e}_2) = 1$. Thus, $\exists~~ a\in\Z$ such that $\set{\x \in \R^2}{x_2 = a} \cap {Q}_4 \neq \emptyset$. Lemma \ref{lem:length} suggests $w_{\set{\x \in \R^2}{x_2 = a} \cap {Q}_4}(\mathbf{e}_1) \geq 1$. It follows that there exists an integer point in ${Q}_4$ and consequently in ${Q}$.

Conversely, assume that $2n_1>(k-1)n_2$. \Cref{lem:lines} yields that $w_{{Q}_5}(\mathbf{e}_2) > 1$. Thus, $\exists~~ a\in\Z$ such that $\set{\x \in \R^2}{x_2 = a} \cap {Q}_5 \neq \emptyset$. Furthermore, by definition of $L_2$, $L_3$ and ${Q}_5$, we have $w_{\set{\x \in \R^2}{x_2 = a}\cap{Q}_5}(\mathbf{e}_1) \geq 1$. It follows that ${Q}_5$ and thus ${Q}$ contains an integer point.
\ignore{Let the region of ${Q}$ between $L_2$ and $L_3$ be ${Q}_5$. Due to \Cref{lem:lines}, $w_{{Q}_5}(\mathbf{e}_2)=1$. Thus, there is a horizontal lattice line, say $l$, through ${Q}_5$. By definition, along any horizontal line between $L_2$ and $L_3$, ${Q}_5$ is at least a unit long. Thus, the length would be at least a unit along $l$. Hence, ${Q}_5$ and thus ${Q}$, contains an integer point.}
\end{proof}

\begin{mainthm}
If $\n$ satisfies $n_2\leq kn_k$ and $n_1~\textit{{\normalfont mod}}((k+1)n_k)\in [n_k,kn_k]$, then it is a lonely runner instance.
\end{mainthm}
\begin{proof}
Since it has been assumed that $n_2\leq kn_k$, it suffices to show the existence of an integral point in the projection of ${P}(\n)$ on to the first coordinate, for the existence of an integral point in ${P}(\n)$. Consider the projection, ${P}_1(\n)$, of ${P}(\textbf{n})$ on $x_1$. In particular, ${P}_1(\n)$ can be described as,
$$
    \dfrac{n_1}{(k+1)n_k}-\dfrac{k}{k+1} \leq x_1 \leq \dfrac{kn_1}{(k+1)n_2}-\dfrac{1}{k+1}
$$
The length, $l$, of the interval is,
\begin{align}
    l & = \left(\dfrac{kn_1}{(k+1)n_2}-\dfrac{1}{k+1}\right)-\left(\dfrac{n_1}{(k+1)n_k}-\dfrac{k}{k+1}\right)\nonumber\\
      & = \dfrac{n_1}{k+1}\left(\dfrac{k}{n_2}-\dfrac{1}{n_k}\right)+\dfrac{k-1}{k+1}\nonumber\\
      & = \dfrac{n_1}{k+1}\dfrac{kn_k-n_2}{n_2n_k}+\dfrac{k-1}{k+1}\nonumber\\
      & \geq \dfrac{k-1}{k+1}. \label{ineq:thm6-1}
\end{align}
where the above inequality follows from $n_2\leq kn_k$ and positivity of $n_1$ and $k$.

Assume that $n_1=a(k+1)n_k+b$, where $a\in \mathbb{N}$ and $0\leq b<(k+1)n_k$. Specifically, $n_k\leq b\leq kn_k$ since $n_1~\textit{{\normalfont mod}}((k+1)n_k)\in [n_k,kn_k]$. With this assumption, we have
\begin{align}
    \dfrac{n_1}{(k+1)n_k}-\dfrac{k}{k+1} & = \dfrac{a(k+1)n_k+b}{(k+1)n_k}-\dfrac{k}{k+1}\nonumber\\
      & = a+\dfrac{b-kn_k}{(k+1)n_k}\nonumber\\
      & = a-1+\dfrac{b+n_k}{(k+1)n_k}\nonumber\\
      & \leq a-1+\dfrac{kn_k+n_k}{(k+1)n_k}\nonumber\\
      & = a. \label{ineq:thm6-2}
\end{align}
Additionally, using \eqref{ineq:thm6-1} we obtain
\begin{align}
    \dfrac{kn_1}{(k+1)n_2}-\dfrac{1}{k+1} & \geq \dfrac{n_1}{(k+1)n_k}-\dfrac{k}{k+1}+\dfrac{k-1}{k+1}\nonumber\\
      & = \dfrac{a(k+1)n_k+b}{(k+1)n_k}-\dfrac{k}{k+1}+\dfrac{k-1}{k+1}\nonumber\\
      & = a+\dfrac{b-kn_k}{(k+1)n_k} +\dfrac{k-1}{k+1}\nonumber\\
      & \geq a+\dfrac{n_k-kn_k}{(k+1)n_k} +\dfrac{k-1}{k+1}\nonumber\\
      & = a. \label{ineq:thm6-3}
\end{align}
From \eqref{ineq:thm6-2} and \eqref{ineq:thm6-3}, we have $\dfrac{n_1}{(k+1)n_k}-\dfrac{k}{k+1}\leq a\leq \dfrac{kn_1}{(k+1)n_2}-\dfrac{1}{k+1}$. The result follows.
\end{proof}

\ignore{It has been proven in \cite{polyhedra2019matthias} that the sets of speeds satisfying $n_1\leq kn_k$ are lonely runner instances. In the following, we provide an alternative proof for the same, by explicitly giving a time at which the conjecture is satisfied.} 

\section{Proofs of Theorems~\ref{thm:6} and~\ref{thm:7}}\label{sec:timeproofs}

\begin{mainthm}
$t=\dfrac{k}{(k+1)n_1}$ is a suitable time for $\n$ if and only if $\n$ satisfies $n_1\leq kn_k$.
\end{mainthm}
\begin{proof}
Assume that $n_1\leq kn_k$. Then for all $i\in [k]$
\begin{gather*}
    \dfrac{1}{k} \leq \dfrac{n_i}{n_1} \leq 1\\
    \dfrac{1}{k+1} \leq \dfrac{k}{(k+1)n_1}n_i \leq \dfrac{k}{k+1}\\
    \left\{\dfrac{1}{k+1}\right\} \leq \left\{\dfrac{k}{(k+1)n_1}n_i\right\} \leq \left\{\dfrac{k}{k+1}\right\}
\end{gather*}
Thus, $\dfrac{k}{(k+1)n_1}$ is a suitable time.\\

Conversely, assume that $n_1>kn_k$. We have,
\begin{gather*}
    0<\dfrac{k}{(k+1)n_1}n_k<\dfrac{1}{k+1}\\
    0<\left\{\dfrac{k}{(k+1)n_1}n_k\right\}<\left\{\dfrac{1}{k+1}\right\}
\end{gather*}
It follows that $\left\{\dfrac{k}{(k+1)n_1}n_k\right\}\notin \left[\dfrac{1}{k+1},\dfrac{k}{k+1}\right]$ and, consequently, $\dfrac{k}{(k+1)n_1}$ is not a suitable time.
\end{proof}

\begin{lem}\label{lem:t1-t}
Given any $\n$, $t\in[0,1]$ is a suitable time iff $1-t$ is.
\end{lem}
\begin{proof}
Let $t$ be a suitable time. It follows that $\{n_it\}=b_i$ for some $b_i\in \left[\dfrac{1}{k+1},\dfrac{k}{k+1}\right]$, $i \in [k]$.

Observe that, since $t \leq 1$, for $i\in [k]$, the $i^{th}$ runner will have covered at most $n_i-1$ rounds. Thus, $n_it=a_i+b_i$ where $a_i\in \mathbb{N}$ and $a_i\leq n_i-1$. Now, consider the total distance covered by the $i^{th}$ runner up to time $1-t$.
\begin{gather*}
n_i(1-t)=n_i-n_it=n_i-(a_i+b_i)=(n_i-a_i-1)+(1-b_i)\\
\{n_i(1-t)\}=\{(n_i-a_i-1)+(1-b_i)\}=\{1-b_i\}.
\end{gather*}
Furthermore, for $i\in [k]$, $b_i\in \left[\dfrac{1}{k+1},\dfrac{k}{k+1}\right]$ implies that $1-b_i\in \left[\dfrac{1}{k+1},\dfrac{k}{k+1}\right]$. 

We have, $\{n_i(1-t)\}\in \left[\dfrac{1}{k+1},\dfrac{k}{k+1}\right]$ for all $i\in [k]$. Thus $1-t$ is a suitable time as well.

A similar argument can be made by assuming $1-t$ to be a suitable time, which shows $1-(1-t) = t$ is a suitable time as well. The result follows.
\end{proof}
\begin{mainthm}
$\n$ is a lonely runner instance iff there is a suitable time $t\leq \dfrac{1}{2}$.
\end{mainthm}
\begin{proof}
Sufficiency follows from the definition of suitable time. To see the necessity, assume that $\n$ is a \emph{lonely runner instance}. Thus there is a suitable time $s\in[0,1]$. Consider the function
\begin{equation*}
t=
\begin{cases}
s & \text{if }s\leq \dfrac{1}{2}\\
1-s & \text{otherwise}
\end{cases}
\end{equation*}
\Cref{lem:t1-t} yields that both $s$ and $1-s$ are suitable times, implying that at least one of $s$ and $1-s$ is at most $\dfrac{1}{2}$. The result follows.
\end{proof}

\section{A new sufficiency condition}\label{sec:sufficiency}
From a polyhedral theory perspective, the existence of an integer lattice point in $P(\n)$ is the best known sufficiency condition to prove the \emph{Lonely Runner Conjecture}. In this section, we provide a new sufficiency condition to prove the conjecture.
\begin{lem}\label{lem:lattice}
Consider $p\in\mathbb{Z}^+$ and $\n\in \mathbb{Z}^k$. Then, the set of points, $\Lambda_{\n,p}$, defined by
\begin{equation*}
\Lambda_{\n,p}:=\left\{\mathbf{x}\in\mathbb{R}^k:\mathbf{x}=\dfrac{a}{p}\n+\mathbf{b},0\leq a\leq p-1, a\in\mathbb{Z}, \mathbf{b}\in\mathbb{Z}^k\right\}
\end{equation*}
is a lattice.
\end{lem}
\begin{proof}
First, let $a=0$ and $\mathbf{b}=\mathbf{0}$. Then, $\mathbf{x}=\mathbf{0}+\mathbf{0}=\mathbf{0}\in\Lambda_{\n,p}$. Next, consider two points $\mathbf{y}=\dfrac{a_1}{p}\n+\mathbf{b_1}$ and $\mathbf{z}=\dfrac{a_2}{p}\n+\mathbf{b_2}$, where $0\leq a_1,a_2\leq p-1$, $a_1$, $a_2\in \mathbb{Z}$ and $\mathbf{b_1}$, $\mathbf{b_2}\in\mathbb{Z}^k$. Their sum is given by $\mathbf{y}+\mathbf{z}=\left(\dfrac{a_1}{p}\n+\mathbf{b_1}\right)+\left(\dfrac{a_2}{p}\n+\mathbf{b_2}\right)=\dfrac{a_1+a_2}{p}\n+(\mathbf{b_1}+\mathbf{b_2})$. Note that $(a_1+a_2)\in\mathbb{Z}$ and $(\mathbf{b_1}+\mathbf{b_2})\in\mathbb{Z}^k$. If $0\leq a_1+a_2\leq p-1$, then $\mathbf{y}+\mathbf{z}\in\Lambda_{\n,p}$. If not, then $p\leq a_1+a_2\leq 2p-2$. Then, $\mathbf{y}+\mathbf{z}=\dfrac{a_1+a_2-p}{p}\n+(\mathbf{b_1}+\mathbf{b_2}+\n)$, and thus, $\mathbf{y}+\mathbf{z}\in\Lambda_{\n,p}$. Finally, let $\mathbf{u},\mathbf{v}\in\Lambda_{\n,p}$. If $\mathbf{u}\neq\mathbf{v}$, then $\exists~i\in[k]~\text{s.t}~|u_i-v_i|\geq \dfrac{1}{p}$. As a result, $\mathcal{B}\left(\mathbf{v},\dfrac{1}{2p}\right)=\mathbf{v}$. Therefore, $\Lambda_{\n,p}$ is a discrete additive subgroup or, equivalently, a lattice.
\end{proof}

\ignore{\begin{prop}\label{thm:sublattice}
If $p_1, p_2, m\in\mathbb{N}^+$ such that $p_2=mp_1$, then $\Lambda_{\n,p_1}\subseteq \Lambda_{\n,p_2}$.
\end{prop}
\begin{proof}
Let $\mathbf{x}=\dfrac{a}{p_1}\n+\mathbf{b}\in\Lambda_{\n,p_1}$. An equivalent representation of $\mathbf{x}$ is as $\mathbf{x}=\dfrac{am}{mp_1}\n+\mathbf{b}=\dfrac{am}{p_2}\n+\mathbf{b}\in\Lambda_{\n,p_2}$. Thus, we have $\mathbf{x}\in\Lambda_{\n,p_2}$.
\end{proof}}

\begin{lem}\label{lem:sublattice}
If $p_1, p_2, m\in\mathbb{Z}^+$ such that $p_2=mp_1$, then $\Lambda_{\n,p_1}\subseteq \Lambda_{\n,p_2}$.
\end{lem}
\begin{proof}
Consider $\mathbf{x}=\dfrac{a}{p_1}\n+\mathbf{b}\in\Lambda_{\n,p_1}$. Note that $\mathbf{x}=\dfrac{am}{mp_1}\n+\mathbf{b}=\dfrac{am}{p_2}\n+\mathbf{b}\in\Lambda_{\n,p_2}$, where $0\leq am\leq m(p_1-1)\leq mp_1-1=p_2-1$. Thus, we have $\mathbf{x}\in\Lambda_{\n,p_2}$ and the result holds.
\end{proof}

The equation of any line in $\mathbb{R}^n$($n\geq 2$) can be represented using the symmetric form. If $(y_1,\ldots,y_n)\in\mathbb{R}^n$ is a point on the line and the direction ratios of the line are given by $(a_1,\ldots,a_n)\in\mathbb{R}^n$, then the equation (in symmetric form) of the line is
\begin{equation*}
\dfrac{x_1-y_1}{a_1}=\dfrac{x_2-y_2}{a_2}=\ldots=\dfrac{x_n-y_n}{a_n}=t
\end{equation*}
where $t\in\mathbb{R}$. Other than in $\mathbb{R}^2$, there is no known condition to determine whether or not a line passes through an integer point. We now provide one such condition for the case where $\mathbf{a}\in\mathbb{Q}^n$ (or equivalently $\mathbb{Z}^n$) and $\mathbf{y}\in\mathbb{Q}^n$. \ignore{Before getting into the details of the result, it is important to \review{note that the line remains the same independent of which of $\mathbf{a}$ and $m\mathbf{a}$, where $m\in\mathbb{R}$, is used as the direction ratios}. Due to this, $\mathbf{a}\in\mathbb{Q}^n$ can be converted to $\mathbf{a}\in\mathbb{Z}^n$.}

\begin{prop}\label{thm:pimplies1}Let $l$ be a line in $\mathbb{R}^n$, with direction ratios $\mathbf{r}=(r_1,\ldots,r_n)\in\mathbb{Z}^n$. If there exists $p\in\mathbb{N}^+$ such that $l\cap \Lambda_{\mathbf{r},p}\neq \phi$, then $l\cap \Lambda_{\mathbf{r},1}\neq\phi$.
\end{prop}
\begin{proof}
Let $\dfrac{a}{p}\mathbf{r}+\mathbf{b}$ be a point on $l$. Then the equation of this line is given by:
\begin{equation*}
\dfrac{x_1-\left(\dfrac{a}{p}r_1+b_1\right)}{r_1}=\dfrac{x_2-\left(\dfrac{a}{p}r_2+b_2\right)}{r_2}=\ldots=\dfrac{x_n-\left(\dfrac{a}{p}r_n+b_n\right)}{r_n}
\end{equation*}
Equating the terms corresponding to $x_i$ and $x_n$ gives
\begin{equation*}
\dfrac{x_i-\left(\dfrac{a}{p}r_i+b_i\right)}{r_i}=\dfrac{x_n-\left(\dfrac{a}{p}r_n+b_n\right)}{r_n}
\end{equation*}
On simplification, the equations can be written as:
\begin{equation*}
x_i=\dfrac{r_i}{r_n}x_n+\dfrac{r_nb_i-r_ib_n}{r_n}
\end{equation*}
On substituting $x_n=b_n+cr_n, c\in\mathbb{Z}$, we get $x_i\in\mathbb{Z}~~\forall 1\leq i\leq n$. Thus, $l$ contains an integer point.
\end{proof}

\begin{prop}\label{thm:parallel}
Any line in $\mathbb{R}^k$ with direction ratios $\n$ is parallel to each of the facets of $P(\n)$.
\end{prop}
\begin{proof}
Consider any facet of $P(\n)$. The normal vector of this hyperplane is given by $\mathbf{N}=n_j\mathbf{e_i}-n_i\mathbf{e_j}$. Then, $\mathbf{N}\cdot \n=n_i(n_j)+n_j(-n_i)=0$. Since the dot product is zero, the normal vector of the hyperplane and the direction vector of the line are parallel to each other. The result follows since the above holds for each of the facets of $P(\n)$.
\end{proof}

\begin{rem}\label{rem:line}
Let $l$ be a line in $\mathbb{R}^k$, with direction ratios $\n$. Then, due to \Cref{thm:parallel}, $l\subseteq P(\n)$ or $l\cap P(\n)=\emptyset$.
\end{rem}
\setcounter{mainthm}{9}
\begin{mainthm}\label{cor:sufficiency}
If $P(\n)\cap\left(\bigcup\limits_{p\geq 1}\Lambda_{\n,p}\right)\neq\emptyset$, then $P(\n)\cap\Lambda_{\n,1}\neq\emptyset$.
\end{mainthm}
\begin{proof}
Consider any line with direction ratios $\n$. Assume that $P(\n)\cap\left(\bigcup\limits_{p\geq 1}\Lambda_{\n,p}\right)\neq\emptyset$. Due to \Cref{thm:pimplies1} and Remark \ref{rem:line}, we have $P(\n)\cap\Lambda_{\n,1}\neq\emptyset$.
\end{proof}

\Cref{cor:sufficiency} serves as a new sufficiency condition that may be used to prove the conjecture. In order to understand the geometry of a lattice, it helps to have knowledge of the basis of the lattice. To this end, we attempt to identify a basis for $\Lambda_{\n,p}$ in the following.

\ignore{As a result, we can expect that for a \emph{good} basis of the lattice, the corresponding fundamental parallelepiped is contained in $\left[\frac{1}{k+1},\frac{k}{k+1}\right]^k$. Consequently, it can be shown that some translate of this fundamental parallelepiped is contained in $P(\n)$.}

\begin{prop}\label{thm:basis}
Let $p$ be such that $GCD(n_i,p)=1$ for some $i\in[k]$. Then, 
$$
B=\left\{\dfrac{m}{p}\mathbf{n}-\mathbf{d},\mathbf{e}_1,\ldots,\mathbf{e}_{i-1},\mathbf{e}_{i+1},\ldots,\mathbf{e}_k\right\}
$$

is a basis of $\Lambda_{\mathbf{n},p}$ where $m=n_i^{-1}~\textit{{\normalfont mod}}~p$ and $\mathbf{d}=\left(\Bigg\lfloor\dfrac{mn_1}{p}\Bigg\rfloor,\ldots,\Bigg\lfloor\dfrac{mn_k}{p}\Bigg\rfloor\right)$.
\end{prop}
\begin{proof}
WLOG assume that $GCD(n_k,p)=1$. From this, we get that $mn_k=1+cp$ where $c\in\mathbb{Z}$. Then, $d_k=c$ and we get,
\begin{equation}
\dfrac{mn_k}{p}-d_k=\dfrac{1}{p}\label{eq:1mod}
\end{equation}
Next, note that $0\leq m<p$. By letting $a=m$ and $\mathbf{b}=\mathbf{d}$ in the definition of $\Lambda_{\mathbf{n},p}$, we get that $\left(\dfrac{m}{p}\mathbf{n}-\mathbf{d}\right)$ is a lattice point. Moreover, it is easy to check that $\mathbf{e}_1,\ldots,\mathbf{e}_{k-1}$ are lattice points.

Now consider an arbitrary lattice point. Let it be $\left(\dfrac{a}{p}\mathbf{n}+\mathbf{b}\right)$. Assume that
\begin{equation}
\dfrac{a}{p}\mathbf{n}+\mathbf{b}=c_1\left(\dfrac{m}{p}\mathbf{n}-\mathbf{d}\right)+c_2\mathbf{e}_1+\ldots+c_k\mathbf{e}_{k-1}\label{eq:thm25-2}
\end{equation}
where $c_1,\ldots,c_k$ are real numbers. To prove the result, it suffices to show that these $c_i$s are, in fact, all integers.

Equating the $k$th coordinates of the left and right hand side of \eqref{eq:thm25-2}, we get:
\begin{equation*}
\dfrac{an_k}{p}+b_k=c_1\left(\dfrac{mn_k}{p}-d_k\right)
\end{equation*}
On simplification, we then get:
\begin{equation*}
c_1 = \dfrac{\frac{an_k}{p}+b_k}{\frac{mn_k}{p}-d_k} = an_k+pb_k
\end{equation*}
where the second equality follows from \eqref{eq:1mod}. It is easy to see that $c_1$ is an integer. Next, we equate the $i$th coordinate of the left and right hand sides of \eqref{eq:thm25-2} for $i<k$. That gives:
\begin{equation*}
\dfrac{an_i}{p}+b_i=\dfrac{mc_1n_i}{p}-c_1d_i+c_{i+1}
\end{equation*}
Substitute the expression obtained for $c_1$ and simplify. We then get:
\begin{equation*}
\dfrac{an_i}{p}+b_i=\dfrac{m(an_k+pb_k)n_i}{p}-(an_k+pb_k)d_i+c_{i+1}=\dfrac{amn_in_k}{p}+mb_kn_i-(an_k+pb_k)d_i+c_{i+1}
\end{equation*}
Writing the equation in terms of $c_{i+1}$ gives:
\begin{align*}
c_{i+1} & = \dfrac{an_i(1-mn_k)}{p}+b_i-mb_kn_i+(an_k+pb_k)d_i\\
& = \dfrac{an_i(-pd_k)}{p}+b_i-mb_kn_i+(an_k+pb_k)d_i\\
& = -an_id_k+b_i-mb_kn_i+(an_k+pb_k)d_i
\end{align*}
where the second equality follows from \eqref{eq:1mod}. As before, $c_{i+1}$ is an integer. Therefore, the result holds.
\end{proof}

Utilizing the basis in \Cref{thm:basis}, the volume of the corresponding fundamental parallelepiped can be computed to be $\dfrac{1}{p}$. Since this volume is decreasing in $p$, one can expect that for a \emph{good} basis of the lattice, the corresponding fundamental parallelepiped is contained in $\left[\frac{1}{k+1},\frac{k}{k+1}\right]^k$. Consequently, it can be shown that some translate of this fundamental parallelepiped is contained in $P(\n)$.

\section{Conclusion}
In this work, we revisited the \emph{Lonely Runner Conjecture}, primarily through a polyhedral perspective. In addition to identifying new families of speed vectors that can be characterized as \emph{lonely runner instances}, we propose a new sufficiency condition for a set of speeds to satisfy the conjecture.\\

There are several ways in which the conjecture has been or is being approached. However, more often than not, we have to deal with infinite sample spaces. We have tried to condense the sample space to a finite set through Conjecture \ref{conj:1}. We hope that this may provide promising new approaches to resolve the long standing conjecture.
\bibliographystyle{plainnat}
\bibliography{references.bib}
\end{document}